\documentclass[11pt]{article}
\usepackage{amsthm, amsmath, amssymb, amsfonts, url, booktabs, tikz, setspace, fancyhdr, amsbsy}
\usepackage[margin = 0.8in]{geometry}
\usepackage{esint}
\usepackage{verbatim}

\newtheorem{theorem}{Theorem}[section]

\newtheorem{lemma}[theorem]{Lemma}

\theoremstyle{definition}
\newtheorem{definition}[theorem]{Definition}

\theoremstyle{remark}

\newcommand{\norm}[1]{\left\lVert#1\right\rVert}

\newcommand{\R}{\mathbb{R}}
\newcommand{\N}{\mathbb{N}}

\newcommand{\defeq}{\mathrel{\mathop:}=}

\newcommand{\uu}{\boldsymbol{u}}
\newcommand{\DD}{\boldsymbol{D}}

\newcommand{\SSS}{\boldsymbol{S}}
\newcommand{\GGG}{\boldsymbol{G}}

\newcommand{\dx}{\,\mathrm{d}x}

\newcommand{\dt}{\,\mathrm{d}t}
\newcommand{\ds}{\,\mathrm{d}s}
\newcommand{\DDD}{\overline{\DD}}
\numberwithin{equation}{section}

\def\ocirc#1{\ifmmode\setbox0=\hbox{$#1$}\dimen0=\ht0
    \advance\dimen0 by1pt\rlap{\hbox to\wd0{\hss\raise\dimen0
    \hbox{\hskip.2em$\scriptscriptstyle\circ$}\hss}}#1\else
    {\accent"17 #1}\fi}

\begin{document}

\title{Temporal decay of strong solutions for generalized Newtonian fluids with variable power-law index}

\author{Seungchan Ko \\ {\textit{\footnotesize{Department of Mathematics, The University of Hong Kong, Pokfulam Road, Hong Kong.}}} \thanks{Email: \tt{scko@maths.hku.hk}}
}

\date{~}

\maketitle

~\vspace{-1.5cm}

\begin{abstract}
We consider the motion of a power-law-like generalized Newtonian fluid in $\R^3$, where the power-law index is a variable function. This system of nonlinear partial differential equations arises in mathematical models of electrorheological fluids. The aim of this paper is to investigate the decay properties of strong solutions for the model, based on the Fourier splitting method. We first prove that the $L^2$-norm of the solution has the decay rate $(1+t)^{-\frac{3}{4}}$. If the $H^1$-norm of the initial data is sufficiently small, we further show that the derivative of the solution decays in $L^2$-norm at the rate $(1+t)^{-\frac{5}{4}}$.
\end{abstract}


\smallskip


\begin{section}{Introduction}
We are interested in studying the decay properties for solutions of a system of nonlinear partial differential
equations (PDEs) modelling the rheological response of electrorheological fluids. The electrorheological fluid is a viscous fluid with a special property: When it is disposed to an electro-magnetic field, the viscosity exhibits a significant change. For instance, there exist some electrorheological fluids whose viscosity changes by a factor of $1000$ as a response to the application of an electric field within 1ms. These days, some useful electrorheological fluids were found with the quality and potential for a wide range of scientific and industrial applications, including for example, clutches, shock absorbers and actuators.
 In this paper, we consider the reduced model for the incompressible electrorheological fluids, which consists of the following system of PDEs:
\begin{alignat}{2}
\partial_t\uu+(\uu\cdot\nabla)\uu-{\rm{div}}\,\SSS(\DD\uu)+\nabla \pi&=0\qquad &&{\rm{in}}\,\,Q_T=(0,T)\times\R^d,\label{eq1}\\
{\rm{div}}\,\uu&=0\qquad &&{\rm{in}}\,\,Q_T=(0,T)\times\R^d,\label{eq2}
\end{alignat}
where the extra stress tensor is of the form
\begin{equation}\label{ST}
\SSS(\DD\uu)=(1+|\DD\uu|^2)^{\frac{p(t,x)-2}{2}}\DD\uu.
\end{equation}
Indeed, the exponent $p(\cdot)$ depends on the magnitude of the electric field $|\boldsymbol{E}|$, which is the solution of quasi-static Maxwell's equations. However, since the equations for $|\boldsymbol{E}|$ decouple from \eqref{eq1}-\eqref{eq2}, we can consider $p(\cdot)$ as a given function and restrict ourselves to the study of the equations \eqref{eq1}-\eqref{ST}.
In the above system of equations, $\uu:Q_T\rightarrow\R^d$, $p:Q_T\rightarrow\R$, denote the velocity field and pressure respectively, and $\DD\uu$ is the symmetric velocity gradient, i.e. $\DD\uu=\frac{1}{2}(\nabla\uu+(\nabla\uu)^T)$. Here we prescribe the initial condition
\[\uu(0,x)=\uu_0(x)\quad{\rm{in}}\,\,\R^d.\]
Such electrorheological models were studied in \cite{R2000, R2004}, where the detailed description of the model is presented and existence theory and numerical approximation are developed. 

In the present paper, we shall investigate the decay properties of the model \eqref{eq1}-\eqref{ST}.  Regarding this matter, for the Navier--Stokes equations, there are a large number of contributions; see, for example, \cite{Schonbek_1, Schonbek_2} where the Fourier splitting method was developed, and \cite{kato, ref_1, ref_2, Schonbek_3, Schonbek_4, titi} with the references therein for related results such as upper and lower bounds on the decay rate for various norms and classes of initial data. For the magnetohydrodynamics equations where the Navier--Stokes equations are coupled with Maxwell's equations, see \cite{mag_1, mag_2, mag_3, mag_4}. On the other hand, for the non-Newtonian fluid flow model, the algebraic $L^2$ decay of solutions was examined in \cite{non_1, non_2, Dong}, and the similar results for the non-Newtonian fluids combined with Maxwell's equations were studied in \cite{non_mag_2, non_mag_1}.

To the best of our knowledge, there is no result for the decay properties of solutions to generalized Newtonian fluids with a variable power-law index. In this paper, we shall study the $L^2$ decay rates for the $L^2$-norm and $H^1$-seminorm of strong solution to the electrorheological fluids.
\end{section}

\begin{section}{Preliminaries and main theorem}
In this section, we first introduce some notations and preliminaries which will be used throughout the paper. For two vectors $\boldsymbol{a}$ and $\boldsymbol{b}$, $\boldsymbol{a}\cdot \boldsymbol{b}$ denotes the scalar product, and $C$ denotes a generic positive constant, which may differ at each appearance. For $1\leq p\leq\infty$ and $k\in\N$, we mean by $W^{k,p}(\R^d)$ the standard Sobolev space and we denote $H^k(\R^d)=W^{k,2}(\R^d)$. Furthermore, for simplicity, we write $\|\cdot\|_p=\|\cdot\|_{L^p(\R^d)}$ and $\|\cdot\|_{k,p}=\|\cdot\|_{W^{k,p}(\R^d)}$. We also recall Korn's inequality (see, for instance, Lemma 2.7 in \cite{Cauchy}).
\begin{lemma}
Assume that $1<q<\infty$. Then there exists a positive constant $C>0$ depending on $q$ such that for any $\uu\in W^{1,q}(\R^d)^d$, we have
\[\|\nabla\uu\|_q\leq C\|\DD\uu\|_q.\]
\end{lemma}

Next, since we deal with the variable power-law index, we introduce  the variable-exponent Lebesgue space. We denote a set of all measurable functions $p:\Omega\rightarrow[1,\infty)$ by $\mathcal{P}(\Omega)$ and we call a function $p\in\mathcal{P}(\Omega)$ a variable exponent. Let $\Omega\subset\R^d$ be open and $p\in\mathcal{P}(\Omega)$. We define the bounds for $p$ as $p^-={\rm{ess}}\,\inf_{x\in\Omega}p(x)$ and $p^+={\rm{ess}}\,\sup_{x\in\Omega}p(x).$
For each measurable functions $f:\Omega\rightarrow\R$, we define the modular $|f|_{p(\cdot)}$ by
\[|f|_{p(\cdot)}=\int_{\Omega}|f(x)|^{p(x)}\dx.\] 
Then we define the variable-exponent Lebesgue spaces with the corresponding Luxembourg norms:
\begin{equation*}
L^{p(\cdot)}(\Omega)\defeq \left\{u\in L^1_{\rm{loc}}(\Omega):|u|_{p(\cdot)}<\infty\right\},\quad{\rm{with}}\quad \norm{u}_{L^{p(\cdot)}(\Omega)}\defeq \inf\left\{\lambda>0:\bigg|\frac{u(x)}{\lambda}\bigg|_{p(\cdot)}\leq1\right\}.
\end{equation*}
If $1<p^-\leq p^+<\infty$, it is straightforward to show that  the variable-exponent space $L^{p(\cdot)}(\Omega)$ is a reflexive, separable Banach space. Throughout this paper, we shall assume that  $1<p^-\leq p^+<\infty$.

Next we introduce the minimum regularity of the exponent $p(\cdot)$, which guarantees the validity of various results from the theory of classical Lebesgue space: {\textit{{\rm{log}}-H\"older continuity}}.

\begin{definition}
We call a function $p:\Omega\rightarrow\R$ is locally {\textit{{\rm{log}}-H\"older continuous}} on $\Omega$ if there exists $C_1$ satisfying for all $x,y\in\Omega$,
\begin{equation}\label{local_log_holder}
|p(x)-p(y)|\leq\frac{C_1}{{\rm{log}}(e+1/|x-y|)}.
\end{equation}
We also say that $p$ satisfies the {\textit{{\rm{log}}-H\"older decay condition}} if there exist constants $p_{\infty}\in\R$ and $C_2>0$ such that for all $x\in\Omega,$
\begin{equation}\label{log_holder_decay}
|p(x)-p_{\infty}|\leq\frac{C_2}{{\rm{log}}(e+|x|)}.
\end{equation}
We say that $p$ is {\textit{globally {\rm{log}}-H\"older continuous}} in $\Omega$ if it satisfies both \eqref{local_log_holder} and \eqref{log_holder_decay}. We call $C_{\rm{log}}(p)\defeq\max\{C_1,C_2\}$ a log-{\textit{H\"older constant}} of $p$.
\end{definition}

\begin{definition}
We define the class of log-H\"older continuous variable exponents:
\[\mathcal{P}^{\rm{log}}(\Omega)\defeq\left\{p\in\mathcal{P}(\Omega):\frac{1}{p}\,\,{\rm{is}}\,\,{\rm{globally}}\,\,{\textrm{log-H\"older}}\,\,{\rm{continuous}}\right\}.\]
If $\Omega$ is unbounded, we define $p_{\infty}$ by $\frac{1}{p_{\infty}}\defeq\lim_{|x|\rightarrow\infty}\frac{1}{p(x)}$.
\end{definition}
Note that, since $p\mapsto\frac{1}{p}$ is bilipschitz form $[p^-,p^+]$ to $\left[\frac{1}{p^+},\frac{1}{p^-}\right]$, if $p\in\mathcal{P}(\Omega)$ and $p^+<\infty$, we can show that $p\in\mathcal{P}^{\rm{log}}(\Omega)$ if and only if $p$ is globally H\"older continuous. For further information, see \cite{DHHR2011} as an extensive source of information for the variable-exponent spaces.

Next, we introduce the notation
\[\overline{\DD}\uu\defeq(1+|\DD\uu|^2)^{\frac{1}{2}},\]
and define the following energies:
\begin{align*}
\mathcal{I}_{p}(\uu)(t) &=\int_{\R^d}(\overline{\DD}\uu)^{p(\cdot)-2}|\DD\uu|^2\dx, \\
\mathcal{J}_{p}(\uu)(t) &=\int_{\R^d}(\overline{\DD}\uu)^{p(\cdot)-2}|\nabla\DD\uu|^2\dx. \\
\end{align*}


In this paper, for the sake of simplicity, we consider the case of three space dimensions. Note that the proof for the general case follows in a similar manner. The existence of strong solutions of the model \eqref{eq1}-\eqref{ST} can be found in some literatures. For example, in \cite{strong_sol_per}, it was shown that the local strong solutions exist under the condition $\frac{7}{5}<p^-\leq p^+\leq 2$ in a bounded domain with periodic boundary conditions. On the other hand, for the case of the whole space $\R^3$, in \cite{Cauchy} the authors considered the model with a constant exponent $p$, and proved the existence of global-in-time strong solutions provided that $p\geq\frac{11}{5}$.  For the variable-exponent case, we can show the existence of strong solutions in $\R^3$ to the equations \eqref{eq1}, \eqref{eq2} with
\begin{equation}\label{xdep}
\SSS(\DD\uu)=(1+|\DD\uu|^2)^{\frac{p(x)-2}{2}}\DD\uu,
\end{equation}
where the exponent $p(\cdot)$ only depends on the spatial variable $x\in\R^3$. Indeed,  if we assume $p\in W^{1,\infty}(\R^3)\cap\mathcal{P}^{\rm{log}}(\R^3)$ with $p^-\geq\frac{11}{5}$, by following the argument presented in \cite{Cauchy} we can show in a straightforward manner that the global strong solutions of \eqref{eq1}, \eqref{eq2} and \eqref{xdep} exist.  Furthermore, based on the method presented in \cite{Malek}, we can also show the existence of global strong solutions to \eqref{eq1}, \eqref{eq2} and \eqref{xdep} provided that $p\in W^{1,\infty}(\R^3)\cap\mathcal{P}^{\rm{log}}(\R^3)$ and $p^->\frac{5}{3}$, under the smallness assumption of the initial data. The existence of global strong solutions of the model \eqref{eq1}-\eqref{ST} in $\R^3$ where $p=p(t,x)$ depends on both time and space variables is still open.

In this paper, we consider the aforementioned global-in-time strong solutions of the model \eqref{eq1}, \eqref{eq2} and \eqref{xdep} with the condition $p^-\geq\frac{11}{5}$. Note however, that our proofs of Theorem \ref{main_thm} and Theorem \ref{second_thm} also work for the space-time-dependent power-law index, and hence we can extend the current results to the case of the equations \eqref{eq1}-\eqref{ST} if the existence of corresponding strong solutions for $p=p(t,x)$ is guaranteed. Here strong solutions of the model \eqref{eq1}, \eqref{eq2} and \eqref{xdep} mean that $\uu\in L^{\infty}((0,T);H^1(\R^3)^3)\cap L^2((0,T); H^2(\R^3)^3)$, $|\nabla\uu|\in L^{p(\cdot)}(Q_T)\cap L^{\infty}((0,T);L^{p(\cdot)}(\R^3))$ and $\partial_t\uu\in L^2(Q_T)^3$ with the following energy inequalities:
\begin{align}
\sup_{0\leq t\leq T}\|\uu(t)\|^2_2+\int_0^T\mathcal{I}_{p}(\uu)(t)\dt &\leq \|\uu_0\|^2_2\label{energy_1},\\
\sup_{0\leq t\leq T}\|\nabla\uu(t)\|^2_2+\int_0^T\mathcal{J}_{p}(\uu)(t)\dt &\leq C(\|\uu_0\|_{H^1})\label{energy_2}.
\end{align}

Now we are ready to state our main theorems. Note here that the condition $p\in\mathcal{P}^{\rm{log}}(\R^3)\cap W^{1,\infty}(\R^3)$ is required for the existence of strong solutions, and not for the proof of Theorem \ref{main_thm}, and hence will be omitted in the statement of the theorem. On the other hand, the condition $p\in W^{1,\infty}(\R^3)$ is needed for the proof of Theorem \ref{second_thm}.
\begin{theorem}\label{main_thm}
Suppose that $\uu_0\in L^1(\R^3)\cap H^1(\R^3)$, and assume that $p^-\geq\frac{11}{5}$. Then for the strong solutions $\uu$ of \eqref{eq1}, \eqref{eq2} and \eqref{xdep} we have
\[\|\uu(t)\|_2\leq C(1+t)^{-\frac{3}{4}}\quad\forall t>0,\]
where the constant $C>0$ depends on the $L^1$ and $H^1$-norms of $\uu_0$.
\end{theorem}

\begin{theorem}\label{second_thm}
Suppose that $\uu_0\in L^1(\R^3)\cap H^1(\R^3)$, and assume that  $p\in W^{1,\infty}(\R^3)$ with $\frac{11}{5}\leq p^-\leq p^+<\frac{8}{3}$. Then there exists a small number $\varepsilon>0$ such that if $\|\uu_0\|_{H^1}<\varepsilon$, then for the strong solutions $\uu$ of \eqref{eq1}, \eqref{eq2} and \eqref{xdep} we have
\[\|\nabla\uu(t)\|_2\leq C(1+t)^{-\frac{5}{4}}\quad\forall t>0,\]
where the constant $C>0$ depends on the $L^1$ and $H^1$-norms of $\uu_0$.
\end{theorem}
\end{section}

\begin{section}{Proof of Theorem \ref{main_thm}}

In this section, we aim to prove Theorem \ref{main_thm}. Note that the following lemma also holds for the case $p=p(t,x)$ and the proof for this general case is exactly same as the proof below by adding $t$ to the exponent. We first define $\GGG(\DD\uu)\defeq\left((\DDD\uu)^{p(x)-2}-1\right)\DD\uu$ and rewrite the equations \eqref{eq1} as
\begin{equation}\label{re_eq}
\partial_t\uu+(\uu\cdot\nabla)\uu-{\rm{div}}\,\GGG(\DD\uu)-\Delta\uu+\nabla \pi=0.
\end{equation}

\begin{lemma}\label{main_est}
Suppose that $p^-\geq\frac{11}{5}$ and $\uu$ is sufficiently smooth. Then there exist positive constants $C_1$, $C_2$ and $C_3$ such that for almost all time $t\in(0,T)$, the following inequality holds:
\begin{itemize}
\item {\rm{(Case\,\,$1$)}} $p^-\geq3$:
\[\int^t_0\int_{\R^3}|\GGG(\DD\uu)|\dx\ds\leq C_1,\]
\item {\rm{(Case\,\,$2$)}} $\frac{11}{5}\leq p^-<3$:
\[\int^t_0\int_{\R^3}|\GGG(\DD\uu)|\dx\ds\leq C_2 + C_3 \left(\int^t_0\|\uu(s)\|_2^{\frac{2\alpha}{2-\beta}}\ds\right)^{\frac{2-\beta}{2}},\]
\end{itemize}
where $\alpha=\frac{7-p^-}{4}$ and $\beta=\frac{5p^--11}{4}$.
\end{lemma}

\begin{proof}
We first note that the inequality $(1+s)^{\alpha}-1\lesssim s+s^{\alpha}$ for $s\geq0$ and $\alpha>0$. Then we have
\begin{align*}
\int_{\R^3}|\GGG(\DD\uu)|\dx
&\lesssim\int_{\R^3}\left((1+|\DD\uu|)^{p(x)-2}-1\right)|\DD\uu|\dx\\
&\lesssim \int_{\R^3}\left(|\DD\uu|+|\DD\uu|^{p(x)-2}\right)|\DD\uu|\dx\\
&=\int_{\R^3}|\DD\uu|^2\dx+\int_{\R^3}|\DD\uu|^{p(x)-1}\dx\\
&\lesssim \mathcal{I}_p(\uu)+\int_{\{|\DD\uu|\geq1\}}|\DD\uu|^{p(x)-1}\dx + \int_{\{|\DD\uu|<1\}}|\DD\uu|^{p(x)-1}\dx\\
&\lesssim \mathcal{I}_p(\uu)+\int_{\R^3}|\DD\uu|^{p(x)}\dx+\int_{\R^3}|\DD\uu|^{p^--1}\dx\\
&\lesssim \mathcal{I}_p(\uu)+\int_{\R^3}|\DD\uu|^{p^--1}\dx.
\end{align*}
It remains to estimate the second term on the right-hand side. For $p^-\geq 3$, by \eqref{energy_1} and the interpolation inequality,
\[\int^t_0\|\nabla\uu(s)\|^{p^--1}_{p^--1}\ds \lesssim\int^t_0\|\nabla\uu(s)\|_2^{\frac{2}{p^--2}}  \|\nabla\uu(s)\|^{\frac{p^-(p^--3)}{p^--2}}_{p^-}\ds  \lesssim \|\nabla\uu\|_{L^2((0,T);L^2)}^{\frac{2}{p^--2}}  \|\nabla\uu\|_{L^{p^-}((0,T);L^{p^-})}^{\frac{p^-(p^--3)}{p^--2}} <\infty.  \]
For $\frac{11}{5}\leq p^-<3$, as shown in \cite{Dong}, by \eqref{energy_2} and Gagliardo--Nirenberg interpolation inequality (see, for example, \cite{GNI})
\begin{align*}
\int^t_0\|\nabla\uu(s)\|^{p^--1}_{p^--1}\ds 
& \lesssim  \int^t_0\|\uu(s)\|_2^{\alpha}\|\nabla^2\uu(s)\|_2^{\beta}\ds\\
& \lesssim \left(\int^t_0\|\uu(s)\|_2^{\frac{2\alpha}{2-\beta}}\ds\right)^{\frac{2-\beta}{2}}\left(\int^t_0\|\nabla^2\uu(s)\|_2^2\ds\right)^{\frac{\beta}{2}}\\
& \lesssim \left(\int^t_0\|\uu(s)\|_2^{\frac{2\alpha}{2-\beta}}\ds\right)^{\frac{2-\beta}{2}},
\end{align*}
where $\alpha=\frac{7-p^-}{4}$ and $\beta=\frac{5p^--11}{4}$.
\end{proof}
Furthermore, we need the following estimate, where $\widehat{f}$ denotes the Fourier transformation of $f$.
\begin{lemma}\label{second_est}
Suppose that $\uu_0\in H^1(\R^3)\cap L^1(\R^3)$ and $p^-\geq\frac{11}{5}$. Then for a strong solution $\uu$ to the equations \eqref{eq1}, \eqref{eq2} and \eqref{xdep}, we have the following:

\item {\rm{(Case\,\,$1$)}} $p^-\geq3$:
\begin{equation}\label{repeat_1}
|\widehat{\uu}(t,\xi)|\leq C|\widehat{\uu}_0(\xi)|+C|\xi|\left(1+\int^t_0\|\uu(s)\|^2_2\ds\right).
\end{equation}
\item {\rm{(Case\,\,$2$)}} $\frac{11}{5}\leq p^-<3$:
\begin{equation}\label{repeat_2}
|\widehat{\uu}(t,\xi)|\leq C|\widehat{\uu}_0(\xi)|+C|\xi|\left(1+\left(\int^t_0\|\uu(s)\|_2^{\frac{2\alpha}{2-\beta}}\ds\right)^{\frac{2-\beta}{2}}+\int^t_0\|\uu(s)\|^2_2\ds\right).
\end{equation}
\end{lemma}

\begin{proof}
If we take the Fourier transformation on \eqref{re_eq}, we have
\begin{equation}\label{ode}
\widehat{\uu}_t+|\xi|^2\widehat{\uu}=F(t,\xi)\quad{\rm{and}}\quad\widehat{\uu}_0(\xi)\defeq\widehat{\uu}(0,\xi)=\widehat{\uu}_0,
\end{equation}
where
\begin{equation}\label{aid_1}
F(t,\xi)\defeq\widehat{\nabla\cdot\GGG}(t,\xi)-\widehat{(\uu\cdot\nabla)\uu}(t,\xi)-\widehat{\nabla\pi}(t,\xi).
\end{equation}
Note that
\begin{equation}\label{ini_bd}
|\widehat{\uu}_0(\xi)|\leq \bigg|\int_{\R^3}e^{-ix\cdot\xi}\uu_0(x)\dx\bigg|\leq\int_{\R^3}|\uu_0(x)|\dx\leq C,
\end{equation}
For the stress tensor term,
\begin{equation}\label{aid_2}
\big|\widehat{\nabla\cdot\GGG}(t,\xi)\big| = \bigg|\int_{\R^3}e^{-ix\cdot\xi}\nabla\cdot\GGG(\DD\uu)\dx\bigg| \leq|\xi|\int_{\R^3}|\GGG(\DD\uu)|\dx.
\end{equation}
Next, by H\"older's inequality with \eqref{eq2}, we have
\begin{equation}\label{aid_3}
\big|\widehat{(\uu\cdot\nabla)\uu}(t,\xi)\big|=\bigg|\int_{\R^3}e^{-ix\cdot\xi}\nabla\cdot(\uu\otimes\uu)\dx\bigg|\leq|\xi|\|\uu(t)\otimes\uu(t)\|_1\leq|\xi|\|\uu(t)\|_2^2.
\end{equation}
Finally, taking divergence operator on \eqref{eq1} gives
\[\Delta\pi=\sum_{i,j}\frac{\partial^2}{\partial x_i \partial x_j}(-\uu_i\uu_j+\GGG(\DD\uu)_{ij}),\]
and therefore, by H\"older's inequality, we have
\begin{equation}\label{aid_4}
|\widehat{\nabla\pi}(t,\xi)|\leq|\xi|\|\GGG(\DD\uu)\|_1+|\xi|\|\uu(t)\otimes\uu(t)\|_1
\leq|\xi|\|\GGG(\DD\uu)\|_1+|\xi|\|\uu(t)\|_2^2.
\end{equation}
Now, it follows from \eqref{ode} that
\[\widehat{\uu}(t,\xi)=e^{-|\xi|^2t}\widehat{\uu}_0(\xi)+\int^t_0F(s,\xi)e^{-|\xi|^2(t-s)}\ds.\]
Therefore, by \eqref{aid_1}-\eqref{aid_4} and Lemma \ref{main_est},  we obtain the desired result.
\end{proof}

\begin{proof}[Proof of Theorem \ref{main_thm}]
We shall use the standard Fourier splitting method to prove the theorem. From the energy inequality and Korn's inequality with the condition $p^-\geq\frac{11}{5}>2$, we have
\[\frac{1}{2}\frac{{\rm{d}}}{\dt}\int_{\R^3}|\uu(t,x)|^2\dx+C\int_{\R^3}|\nabla\uu(t,x)|^2\dx\leq0.\]
By Plancherel's theorem, it follows that
\[\frac{1}{2}\frac{{\rm{d}}}{\dt}\int_{\R^3}|\widehat{\uu}(t,\xi)|^2{\rm{d}}\xi+C\int_{\R^3}|\xi|^2|\widehat{\uu}(t,\xi)|^2{\rm{d}}\xi\leq0.\]
Next, let us assume that $f(t)$ is a smooth function with $f(0)=1$, $f(t)>0$ and $f'(t)>0$. Then for some constant $C_0>0$, we have
\begin{equation}\label{plan}
\frac{{\rm{d}}}{\dt}\left(f(t)\int_{\R^3}|\widehat{\uu}(t,\xi)|^2{\rm{d}}\xi\right)+C_0f(t)\int_{\R^3}|\xi|^2|\widehat{\uu}(t,\xi)|^2{\rm{d}}\xi \leq f'(t)\int_{\R^3}|\widehat{\uu}(t,\xi)|^2{\rm{d}}\xi.
\end{equation}
If we define the set $L(t)=\{\xi\in\R^3:C_0|\xi|^2f(t)\leq f'(t)\}$ where $C_0>0$ is the constant appearing in \eqref{plan}, we obtain
\[C_0f(t)\int_{\R^3}|\xi|^2|\widehat{\uu}(t,\xi)|^2{\rm{d}}\xi\geq C_0f(t)\int_{L(t)^{c}}|\xi|^2|\widehat{\uu}(t,\xi)|^2{\rm{d}}\xi\geq f'(t)\int_{\R^3}|\widehat{\uu}(t,\xi)|^2{\rm{d}}\xi-f'(t)\int_{L(t)}|\widehat{\uu}(t,\xi)|^2{\rm{d}}\xi,\]
and therefore, from \eqref{plan}, we deduce
\[\frac{\rm{d}}{\dt}\left( f(t)\int_{\R^3}|\widehat{\uu}(t,\xi)|^2{\rm{d}}\xi\right)\leq f'(t)\int_{L(t)}|\widehat{\uu}(t,\xi)|^2{\rm{d}}\xi.\]
Integrating the above inequality over $(0,t)$ yields
\begin{equation}\label{mid_est}
f(t)\int_{\R^3}|\widehat{\uu}(t,\xi)|^2{\rm{d}}\xi\leq\int_{\R^3}|\widehat{\uu}_0(\xi)|^2{\rm{d}}\xi+\int^t_0 f'(s)\int_{L(s)}|\widehat{\uu}(s,\xi)|^2{\rm{d}}\xi\ds.
\end{equation}
\item {\rm{(Case\,\,$1$)}} $p^-\geq3$:
Now we set $f(t)=(1+t)^3$. Then by Lemma \ref{second_est} with \eqref{ini_bd}, \eqref{mid_est} and the energy inequality \eqref{energy_1}, we obtain that
\begin{align*}
(1+t)^3\int_{\R^3}|\widehat{\uu}(t,\xi)|^2{\rm{d}}\xi
& \leq\int_{\R^3}|\widehat{\uu}_0(\xi)|^2{\rm{d}}\xi+C\int^t_0(1+s)^2\int_{L(s)}|\widehat{\uu}_0(\xi)|^2{\rm{d}}\xi\ds\\
& \hspace{5mm}+C\int^t_0(1+s)^2\int_{L(s)}|\xi|^2\left(1+\int^s_0\|\uu(\tau)\|^2_2{\rm{d}}\tau\right)^2{\rm{d}}\xi\ds\\
& \leq C+C\int^t_0(1+s)^{\frac{1}{2}}\ds+C\int^t_0(1+s)^{-\frac{1}{2}}\ds +C\int^t_0(1+s)^{\frac{3}{2}}\ds\\
& \leq C+C(1+t)^{\frac{3}{2}}+C(1+t)^{\frac{1}{2}}+C(1+t)^{\frac{5}{2}}.
\end{align*}
From Plancherel's theorem, we have
\begin{equation}\label{intermed}
\|\uu(t)\|^2_2\leq C(1+t)^{-\frac{1}{2}}.
\end{equation}
Now if we substitute \eqref{intermed} into \eqref{repeat_1}, and repeat the same process, we finally obtain the desired estimate.
\item {\rm{(Case\,\,$2$)}} $\frac{11}{5}\leq p^-<3$:
In this case, we first note that $\frac{1}{2}<\frac{2-\beta}{2}<1$ and $\frac{4\alpha}{2-\beta}>2$. Then by H\"older's inequality and \eqref{energy_1}, we obtain
\begin{align*}
\int^t_0(  &  1+s)^2\int_{L(s)}|\xi|^2\left(1+\left(\int^s_0\|\uu(\tau)\|_2^{\frac{2\alpha}{2-\beta}}{\rm{d}}\tau\right)^{\frac{2-\beta}{2}}+\int^s_0\|\uu(\tau)\|^2_2{\rm{d}}\tau\right)^2{\rm{d}}\xi\ds\\
&\leq C\int^t_0(1+s)^2\int_{L(s)}|\xi|^2\,{\rm{d}}\xi\ds + C\int^t_0(1+s)^2\int_{L(s)}|\xi|^2\left(\int^s_0\|\uu(\tau)\|_2^{\frac{2\alpha}{2-\beta}}\,{\rm{d}}\tau\right)^{2-\beta}\,{\rm{d}}\xi\ds\\
&\hspace{5mm}+C\int^t_0(1+s)^2\int_{L(s)}|\xi|^2\left(\int^s_0\|\uu(\tau)\|^2_2\,{\rm{d}}\tau\right)^2\,{\rm{d}}\xi\ds\\
&\leq C\int^t_0(1+s)^2\int_{L(s)}|\xi|^2\,{\rm{d}}\xi\ds + C\int^t_0(1+s)^2\int_{L(s)}|\xi|^2s^{\frac{2-\beta}{2}}\left(\int^s_0\|\uu(\tau)\|_2^{\frac{4\alpha}{2-\beta}}\,{\rm{d}}\tau\right)^{\frac{2-\beta}{2}}\,{\rm{d}}\xi\ds\\
&\hspace{5mm}+C\int^t_0(1+s)^2\int_{L(s)}|\xi|^2s\left(\int^s_0\|\uu(\tau)\|^4_2\,{\rm{d}}\tau\right)\,{\rm{d}}\xi\ds\\
&\leq C\int^t_0(1+s)^2\int_{L(s)}|\xi|^2\,{\rm{d}}\xi\ds +C\int^t_0(1+s)^2\int_{L(s)}|\xi|^2s^{\frac{2-\beta}{2}}\left(\int^t_0\|\uu(\tau)\|_2^2\,{\rm{d}}\tau+C\right)\,{\rm{d}}\xi\ds\\
&\hspace{5mm}+C\int^t_0(1+s)^2\int_{L(s)}|\xi|^2s\left(\int^t_0\|\uu(\tau)\|^2_2\,{\rm{d}}\tau\right)\,{\rm{d}}\xi\ds\\
&\leq C(1+t)^{\frac{3}{2}}+C(1+t)^{\frac{3}{2}}\left(\int^t_0\|\uu(\tau)\|^2_2\,{\rm{d}}\tau\right).
\end{align*}
Now, by again from Lemma \ref{second_est}, \eqref{ini_bd} and \eqref{mid_est} with $f(t)=(1+t)^3$, we have that
\[(1+t)^3\|\uu(t)\|^2_2 = (1+t)^3\int_{\R^3}|\widehat{\uu}(t,\xi)|^2{\rm{d}}\xi \leq C(1+t)^{\frac{3}{2}}+C(1+t)^{\frac{3}{2}}\left(\int^t_0\|\uu(\tau)\|^2_2\,{\rm{d}}\tau\right).\]
This yields
\[(1+t)^{\frac{3}{2}}\|\uu(t)\|^2_2\leq C+C\int^t_0(1+\tau)^{\frac{3}{2}}\|\uu(\tau)\|^2_2(1+\tau)^{-\frac{3}{2}}\,{\rm{d}}\tau,\]
and therefore, by Gronwall's inequality, we obtain the desired decay estimate.
\end{proof}
\end{section}

\begin{section}{Proof of Theorem \ref{second_thm}}
We begin with the following a priori estimate.
\begin{lemma}\label{small_ode}
Assume that $p\in W^{1,\infty}(\R^3)$ with $p^+<\frac{8}{3}$. Then there exists a small number $\delta>0$ such that if
\begin{equation}\label{small_ass}
\sup_{0\leq t\leq T}\|\uu(t)\|_{H^1} < 2\delta,
\end{equation}
we have the following differential inequality: for almost all time $t\in(0,T)$,
\begin{equation}\label{diff_eq}
\frac{\rm{d}}{\dt}\|\nabla\uu(t)\|^2_2+\|\nabla^2\uu(t)\|^2_2\leq 0.
\end{equation}
\end{lemma}
\begin{proof}
We shall derive some formal inequalities which are essential for the correct arguments. For the detailed arguments, see for example, \cite{strong_sol_per, Malek}. We first differentiate \eqref{eq1} formally with respect to the spatial variable $x_j$ and take scalar product with $\frac{\partial\uu}{\partial x_j}$. Summing over $j=1,2,3$ yields the following a priori estimate (see \cite{strong_sol_per}):
\begin{align*}
\frac{1}{2}\frac{\rm{d}}{\dt}\|&\nabla\uu\|^2_2+\|\nabla^2\uu\|^2_2+\mathcal{J}_p(u)\\
& \lesssim\int_{\R^3}(\uu\cdot\nabla)\uu\cdot\Delta\uu\dx+\int_{\R^3}|\nabla p|(\DDD\uu)^{p(x)-2}\log(\DDD\uu)|\DD\uu||\nabla^2\uu|\dx\\
& \lesssim \|\uu\|_3\|\nabla\uu\|_6\|\nabla^2\uu\|_2+\int_{\R^3}\log(\DDD\uu)|\DD\uu||\nabla^2\uu|\dx\\
&\,\,\,\,\,+\int_{\R^3}|\DD\uu|^{p^+-2}\log(\DDD\uu)|\DD\uu||\nabla^2\uu|\dx\\
& \defeq {\rm{I}}_1+{\rm{I}}_2+{\rm{I}}_3.
\end{align*}
Note that the logarithmic term appears above when we differentiate the stress tensor with the variable exponent $p(x)$. By the interpolation inequality and the Sobolev embedding,
\[{\rm{I}}_1\lesssim\|\uu\|_2^{\frac{1}{2}}\|\uu\|_6^{\frac{1}{2}}\|\nabla\uu\|_6\|\nabla^2\uu\|_2\lesssim\|\uu\|^{\frac{1}{2}}_2\|\nabla\uu\|^{\frac{1}{2}}_2\|\nabla^2\uu\|^2_2.\]
Next, by the inequality $\log(\DDD\uu)\leq C_{\alpha}|\DD\uu|^{\alpha}$ for any $0<\alpha\leq 1$, we have
\[{\rm{I}}_2\lesssim\int_{\R^3}|\DD\uu|^{\frac{2}{3}}|\DD\uu||\nabla^2\uu|\dx\lesssim\||\DD\uu|^{\frac{2}{3}}\|_3\|\DD\uu\|_6\|\nabla^2\uu\|_2\lesssim\|\nabla\uu\|_2^{\frac{2}{3}}\|\nabla^2\uu\|^2_2.\]
Finally, due to the condition $p^+<\frac{8}{3}$,
\[{\rm{I}}_3\lesssim\int_{\R^3}|\DD\uu|^{\frac{2}{3}}|\DD\uu||\nabla^2\uu|\dx\lesssim\|\nabla\uu\|^{\frac{2}{3}}_2\|\nabla^2\uu\|^2_2.\]
Altogether, we conclude that there exist positive constants $C_1$ and $C_2$ such that
\[\frac{\rm{d}}{\dt}\|\nabla\uu\|^2_2+\|\nabla^2\uu\|^2_2\leq C_1\|\uu\|^{\frac{1}{2}}_2\|\nabla\uu\|^{\frac{1}{2}}_2\|\nabla^2\uu\|^2_2+C_2\|\nabla\uu\|^{\frac{2}{3}}_2\|\nabla^2\uu\|^2_2,\]
and hence
\[\frac{\rm{d}}{\dt}\|\nabla\uu\|^2_2+(1-C_1\|\uu\|^{\frac{1}{2}}_2\|\nabla\uu\|^{\frac{1}{2}}_2-C_2\|\nabla\uu\|^{\frac{2}{3}}_2)\|\nabla^2\uu\|^2_2\leq0.\]
Therefore, we obtain the desired inequality if $\sup_{0\leq t \leq T}\|\uu(t)\|_{H^1}<2\delta$ for sufficiently small $\delta>0$.
\end{proof}

\begin{lemma}\label{smallness}
Assume that $p\in W^{1,\infty}(\R^3)$ with $\frac{11}{5}\leq p^-\leq p^+<\frac{8}{3}$. Then there exists a small number $\varepsilon>0$ such that if $\|\uu_0\|_{H^1}<\varepsilon$, we have
\begin{equation}\label{small_verify}
\sup_{0\leq t \leq T}\|\uu(t)\|_{H^1} < 2\delta,
\end{equation}
where $\delta>0$ is the constant defined in Lemma \ref{small_ode}.
\end{lemma}

\begin{proof}
With the same argument as above and by H\"older's inequality and Young's inequality, we have 
\begin{align*}
\frac{1}{2}\frac{\rm{d}}{\dt}\|&\nabla\uu\|^2_2+\|\nabla^2\uu\|^2_2+\mathcal{J}_p(\uu)\\
& \leq\|\nabla\uu\|^3_3+C\int_{\R^3}(\DDD\uu)^{p(x)-2}\log(\DDD\uu)|\DD\uu||\nabla^2\uu|\dx\\
& \leq\|\nabla\uu\|^3_3+\varepsilon\mathcal{J}_p(u)+C\int_{\R^3}(\DDD\uu)^{p^+}\log^2(\DDD\uu)\\
& \leq\|\nabla\uu\|^3_3+\varepsilon\mathcal{J}_p(u)+C\int_{\R^3}\log^2(\DDD\uu)\dx+C\int_{\R^3}|\DD\uu|^{\frac{8}{3}}\log^2(\DDD\uu)\dx\\
& \leq C\|\nabla\uu\|^2_2+C\|\nabla\uu\|^3_3+\varepsilon\mathcal{J}_p(u).
\end{align*}
Next, for $\frac{11}{5}\leq p^-$, we have that (see \cite{Malek})
\[\|\nabla\uu\|^3_3 \leq C_{\varepsilon}\|\nabla\uu\|^{p^-}_{p^-}\|\nabla\uu\|^2_2+\varepsilon\mathcal{J}_{p^-}(\uu)\leq C_{\varepsilon}\|\nabla\uu\|^{p^-}_{p^-}\|\nabla\uu\|^2_2+\varepsilon\mathcal{J}_{p}(\uu).\]
Therefore, we finally have
\[\frac{\rm{d}}{\dt}\|\nabla\uu\|^2_2\leq C(1+\|\nabla\uu\|^{p^-}_{p^-})\|\nabla\uu\|^2_2,\]
which yields by Gronwall's inequality
\[\|\nabla\uu\|^2_2\leq\|\nabla\uu_0\|^2_2\exp\left(\int^t_0(1+\|\nabla\uu(s)\|^{p^-}_{p^-})\ds\right).\]
Thanks to \eqref{energy_1}, $\|\nabla\uu(t)\|^{p^-}_{p^-}$ is integrable with respect to time, and hence, there exists small $T^*>0$ such that
\begin{equation}\label{sol_ini_est}
\sup_{0\leq t \leq T^*}\|\nabla\uu\|^2_2\leq 2\|\nabla\uu_0\|^2_2.
\end{equation} 
Now, suppose that $\|\uu_0\|_{H^1}<\frac{\delta}{\sqrt{2}}$. Then by \eqref{sol_ini_est} and \eqref{energy_1},
\begin{equation}\label{real_small}
\sup_{0\leq t \leq T^*}\|\uu\|_{H^1}<2\delta.
\end{equation}
Then by Lemma \ref{small_ode},
\[\frac{\rm{d}}{\dt}\|\nabla\uu\|^2_2\leq0,\]
which implies that
\begin{equation}\label{mid_step_1}
\|\nabla\uu(T^*)\|^2_2\leq\sup_{0\leq t \leq T^*}\|\nabla\uu\|^2_2\leq\|\nabla\uu_0\|^2_2<\frac{\delta^2}{2}.
\end{equation}
Next, we consider the original problem \eqref{eq1}--\eqref{eq2} for $t\geq T^*$ with the initial data $\uu(T^*)$. With the same argument as above,
\[\sup_{T^*\leq t \leq 2T^*}\|\nabla\uu\|^2_2\leq 2\|\nabla\uu(T^*)\|^2_2<\delta^2,\]
and hence
\[\sup_{T^*\leq t \leq 2T^*}\|\uu\|_{H^1}<2\delta.\]
By Lemma \ref{small_ode} again, we have
\[\|\nabla\uu(2T^*)\|^2_2\leq\sup_{T^*\leq t \leq 2T^*}\|\nabla\uu\|^2_2\leq\|\nabla\uu(T^*)\|^2_2<\frac{\delta^2}{2}.\]
If we repeat the same process for $(n-1)T^*<t<nT^*$ with $n\in\mathbb{N}$, we finally obtain that
\[\sup_{0\leq t \leq T}\|\uu\|_{H^1}<2\delta.\]
\end{proof}

\begin{proof}[Proof of Theorem \ref{second_thm}]
Let $L(t)=\{\xi\in\R^3:|\xi|\leq f(t)\}$ where $f(t)=\left(\frac{1}{1+t}\right)^{\frac{1}{2}}$.
By Plancherel's theorem, we have
\begin{align*}
\|\nabla^2\uu(t)\|^2_2
& = \int_{\R^3}|\xi|^4|\widehat{\uu}(t,\xi)|^2{\rm{d}}\xi\geq |f(t)|^2\int_{L(t)^c}|\xi|^2|\widehat{\uu}(t,\xi)|^2{\rm{d}}\xi\\
& =|f(t)|^2\int_{\R^3}|\xi|^2|\widehat{\uu}(t,\xi)|^2{\rm{d}}\xi-|f(t)|^2\int_{L(t)}|\xi|^2|\widehat{\uu}(t,\xi)|^2{\rm{d}}\xi\\
& \geq |f(t)|^2\|\nabla\uu(t)\|^2_2-|f(t)|^4\int_{L(t)}|\widehat{\uu}(t,\xi)|^2{\rm{d}}\xi\\
& \geq |f(t)|^2\|\nabla\uu(t)\|^2_2-|f(t)|^4\|\uu(t)\|^2_2.\\
\end{align*}
From Lemma \eqref{small_ode} and \eqref{smallness}, we can deduce that for almost all $t\in(0,T)$,
\[\frac{\rm{d}}{\rm{d}t}\|\nabla\uu(t)\|^2_2+\|\nabla^2\uu(t)\|^2_2\leq0.\]
Therefore, we obtain
\[\frac{\rm{d}}{\rm{d}t}\|\nabla\uu(t)\|^2_2+\frac{1}{1+t}\|\nabla^2\uu(t)\|^2_2\leq \left(\frac{1}{1+t}\right)^2\|\uu(t)\|^2_2.\]
Then for $\ell>\frac{5}{2}$, by Theorem \ref{main_thm}
\[\frac{\rm{d}}{{\rm{d}}t}\left((1+t)^{\ell}\|\nabla\uu(t)\|^2_2\right)\leq (1+t)^{\ell-2}\|\uu(t)\|^2_2\leq C(1+t)^{\ell-2-\frac{3}{2}}.\]
By integrating the above inequality over time, we have the desired decay estimate.
\end{proof}

\end{section}


\bibliography{references}
\bibliographystyle{abbrv}


\end{document}